\documentclass[reqno]{amsart}

\usepackage{amsmath}
\usepackage{graphicx}
\usepackage{amsthm}
\usepackage{amsfonts}
\usepackage{amssymb,bbm}
\usepackage[numbers, square]{natbib}
\usepackage{url}
\usepackage{enumerate}
\usepackage[pdftex,bookmarks=true]{hyperref}
\usepackage{xcolor}
\usepackage{comment}

\newcommand\rank{\operatorname{rank}}

\newcommand\R{{\mathbb{R}}}
\newcommand\C{{\mathbb{C}}}

\renewcommand\P{\mathbf{P}}
\newcommand\E{{\mathbf{E}}}

\newcommand\eps{{\varepsilon}}

\newcommand{\cE}{\mathcal{E}}

\newcommand{\toprobab}{\overset{P}{\underset{n\to\infty}\longrightarrow}}

\newcommand{\toas}{\overset{a.s.}{\underset{n\to\infty}\longrightarrow}}

\newcommand{\cA}{\mathcal{A}}

\renewcommand\Re{\operatorname{Re}}
\renewcommand\Im{\operatorname{Im}}

\def\rank{\text{rank}}

\theoremstyle{plain}
\newtheorem{theorem}{Theorem}[section]
\newtheorem{conjecture}[theorem]{Conjecture}

\newtheorem{fact}[theorem]{Fact}
\newtheorem{lemma}[theorem]{Lemma}

\newtheorem{claim}[theorem]{Claim}
\newtheorem*{claim*}{Claim}

\theoremstyle{remark}
\newtheorem{remark}[theorem]{Remark}

\theoremstyle{definition}

\begin{document}

\title[Zeros of a growing number of derivatives of random polynomials with independent roots]{Zeros of a growing number of derivatives of random polynomials with independent roots}

\author{Marcus Michelen}
\address{Department of Mathematics, Statistics and Computer Science,  University of Illinois, Chicago, USA} 
\email{michelen@uic.edu, michelen.math@gmail.com}

\author{Xuan-Truong Vu}
\address{Department of Mathematics, Statistics and Computer Science,  University of Illinois, Chicago, USA}
\email{tvu25@uic.edu, truongvu.math@gmail.com}

\begin{abstract}
 Let $X_1,X_2,\ldots$ be independent and identically distributed random variables in $\C$ chosen from a probability measure $\mu$ and define the random polynomial
 \begin{align*}
 		P_n(z)=(z-X_1)\ldots(z-X_n)\,.
 \end{align*}
We show that for any sequence $k = k(n)$ satisfying $k \leq \log n / (5 \log\log n)$, the zeros of the $k$th derivative of $P_n$ are asymptotically distributed according to the same measure $\mu$.  This extends work of Kabluchko, which proved the $k = 1$ case, as well as Byun, Lee and Reddy who proved the fixed $k$ case.  
\end{abstract}

\maketitle
\section{Introduction} 

Let $\mu$ be a probability measure on $\C$. 
Suppose that $X_1,X_2,\ldots$ are i.i.d.\ random variables with values in $\C$ sampled from $\mu$, and for each $n$ define the random polynomial
\begin{equation*}
	P_n(z)=(z-X_1)\ldots(z-X_n)\,. 
\end{equation*}
By the law of large numbers if we consider the \emph{empirical measure} of $P_n$, where we a put a point mass of $1/n$ at each root of $P_n$, then we see that the empirical measure converges to $\mu$ as $n$ approaches infinity.     Pemantle and Rivin \cite{pemantle_rivin}  conjectured that the same holds for the derivative $P_n$.  To make this precise, we define 	$\mu_n^{(1)} $ to be the probability measure on $\C$ that puts a point mass at each critical point of $P_n$:
\begin{equation*}
	\mu_n^{(1)} := \frac {1}{n-1}\sum_{z\in \C: P_n^{'}(z)=0} \delta_z.
\end{equation*}
Pemantle and Rivin conjectured that $\mu_n^{(1)} \to \mu$ in distribution as $n\to \infty$ and proved their conjecture under the assumption that $\mu$ has finite $1$-energy. Subramanian \cite{subramanian} proved the Pemantle-Rivin conjecture in the special case when $\mu$ is supported on the unit circle.  In an influential work, Kabluchko  \cite{Kabluchko2015} confirmed Pemantle and Rivin's conjecture for \emph{all} probability measures $\mu$.  Since then, attention has been focused on higher derivatives.  To this end, for each $k$ define the (random) probability measure  $\mu_n^{(k)}$ via

\begin{equation*}
	\mu_n^{(k)} := \frac 1{n-k}\sum_{z\in \C: P_n^{(k)}(z)=0} \delta_z\,.
\end{equation*}

Byun, Lee and Reddy \cite{byun-lee-reddy} extended Kabluchko's work and showed that for each fixed $k$,\footnote{The work \cite{byun-lee-reddy} does not state an explicit rate at which $k$ can be taken to grow.  An inspection of the proof shows that it depends on the Levy concentration of the non-atomic part of the random variable $(z - X)^{-1}$ as $z$ varies in $\C$ (see \cite[eq.\ (4.6)]{byun-lee-reddy}).}  we have $\mu_n^{(k)} \to \mu$.  
Looking towards very high derivatives, O'Rourke and Steinerberger \cite{o2021nonlocal} conjecture that for each fixed $t \in [0,1]$, one has that the random measure $\mu_n^{\lfloor t n\rfloor}$ converges to a deterministic measure $\mu_t$.  In the case when the underlying measure $\mu$ is radial, O'Rourke and Steinerberger conjecture that the logarithmic potentials of the limiting measure $(\mu_t)_t$ satisfy a certain partial differential equation.  This partial differential equation has been studied by analysts \cite{alazard2022dynamics,kiselev2022flow}, and O'Rourke and Steinerberger's conjecture was proven in the special case when $\mu$ has real support \cite{hoskins2022semicircle} (see also \cite{hoskins2021dynamics,kabluchko2021repeated}).   In the $t = 0$ case, the prediction of O'Rourke and Steinerberger suggests that the limiting measure should be the same as the underlying measure $\mu$.  

We confirm this in the case that $k$ grows slightly slower than logarithmically:
\begin{theorem}\label{thm: maintheorem}
	Let $\mu$ be a probability measure on $\C$ and $k = k(n)$ be a sequence satisfying $k \leq \frac{\log n}{5 \log \log n}$. Then the sequence $\mu_n^{(k)}\to \mu$ in probability as $n\to\infty$.
\end{theorem}

Further, we show that the case of $k = o(n)$ would follow from an anti-concentration conjecture for elementary symmetric polynomials evaluated at i.i.d.\ random variables (see Conjecture \ref{conj:anti-con} and Remark \ref{remark}).

Our proof of Theorem \ref{thm: maintheorem} takes inspiration from the potential-theoretic approach of Kabluchko's proof of the $k=1$ case (as does the work of Byun, Lee and Reddy \cite{byun-lee-reddy}); the key new step is an anti-concentration ingredient. We take a moment to sketch our proof here. 

The starting point is the following classical fact from potential theory: if $f$ is an analytic function not identically equal to zero then 
\begin{equation}\label{eq:laplace_formula0}
	\frac 1 {2\pi} \Delta \log \left| f \right|= \sum_{\zeta\in \C: f(\zeta)=0} \delta_\zeta
\end{equation}
where the Laplacian is interpreted in the distributional sense.

In particular, if we define
\begin{equation}\label{eq:def_Lnk}
	L_n^{(k)}(z):= \frac{P_n^{(k)}(z)}{k! P_n(z)}=\sum_{1\le i_1<i_2<\dots<i_{k}\le n }Y_{i_1}Y_{i_2}\dots Y_{i_{k}},
\end{equation}
where $Y_i:=\frac{1}{z-X_{i}}$ for $i\in [n]$, then we have
\begin{align}\label{eq:laplace_formula_LnK}
	\frac {1} {2\pi n} \Delta \log \left| L_n^{(k)}(z) \right|&= \frac{1}{n}\sum_{z\in \C: P_n^{(k)}(z)=0} \delta_z- \frac{1}{n}\sum_{z\in \C: P_n(z)=0} \delta_z= \frac{n-k}{n}\mu_n^{(k)}- \mu_n^{(0)}.
\end{align}
By the law of large numbers, the measure $ \mu_n^{(0)}$ tends to $\mu$ as $n\to \infty$, and so in order to show $\mu_n^{(k)}$ converges to $\mu$ it will be enough to show that $\frac{1}{n}\log|L_n^{(k)}(z)|$ goes to $0$ in a sufficiently strong sense to guarantee that $\frac{1}{n}\Delta \log|L_n^{(k)}(z)|$  tends to $0$.

Obtaining an upper bound on the magnitude of $\log|L_n^{(k)}(z)|$ involves controlling two different events: when $|L_n^{(k)}(z)|$ is large and when $|L_n^{(k)}(z)|$ very small. The former event will not be too difficult to deal with, but the latter is trickier.  This comes down to an anti-concentration problem for elementary symmetric polynomials evaluated at i.i.d.\ random variables.  

\begin{lemma}\label{lem:anti-concentration_multiVar}
	Let $Y_1, Y_2, \dots, $ be i.i.d.\ copies of a complex-valued non-degenerate random variable. Then for $1\leq k \leq \frac{\log n}{5\log\log n}$ and each $\eps>0$ we have 
	\begin{align*}
		\lim\limits_{n\rightarrow \infty}\P\left(\left|\sum_{1\le i_1<i_2<\dots<i_k\le n }Y_{i_1}Y_{i_2}\dots Y_{i_k}\right|\le e^{-\eps n}\right)=0.
	\end{align*}
	
\end{lemma}
Here we think of $Y_i=(z-X_i)^{-1}$ as above.   We will deduce Lemma \ref{lem:anti-concentration_multiVar} from a theorem of Meka, Nguyen and Vu \cite{Meka-Ng-Vu2016} concerning anti-concentration of multi-affine polynomials of Bernoulli random variables.  

After showing pointwise bounds on $\log|L_n^{(k)}(z)|$, we upgrade these bounds by showing that they hold in some uniform sense. To take care of this, we again follow Kabluchko and lean on a lemma from Tao and Vu \cite{tao_vu} stating that tightness of an $L^2$-norm will be sufficient to deduce convergence in probability of the measures. For this, we use the Poisson-Jensen formula to relate values of $\log|L_n^{(k)}(z)|$ in a disk to its values on the boundary of a larger disk, and show uniform bounds at the origin and on the boundary of this disk.
\subsection{Notation}
Throughout, we use $B_r(z)$ to be the disk of radius $r$ centered at $z$ and abbreviate $B_r(0)$ as $B_r$. We write $m$ for the standard ($2$-real-dimensional) Lebesgue measure on $\C$. We write $\log z=\log_{+}z-\log_{-}z$ where
\[ \log_{-} z = \begin{cases} 
	|\log z|, & 0 \leq z \leq 1, \\
	0, & z \geq 1, 
\end{cases}
\quad\text {and }\quad
\log_{+} z = \begin{cases} 
	0, & 0 \leq z \leq 1, \\
	\log z, & z \geq 1,
\end{cases} \]
with the convention $\log_{-}0=+\infty$.

\section{Proof of the main result}
We will show two main Lemmas. First we show pointwise convergence of $\frac{1}{n}\log|L_n^{(k)}(z)|$ almost everywhere:
\begin{lemma}\label{lem:Lnk_toprobab_0}
	There is a set $F \subset \C$ with $m(F) = 0$ so that for every $z\in \C\backslash F$ and all sequences $k \leq \frac{\log n}{5 \log \log n}$ we have 
	\begin{equation}\label{eq:log_Lnk_toprobab_0}
		\frac{1}{n}\log |L_n^{(k)}(z)|\toprobab 0.
	\end{equation}
	
\end{lemma}

In order to prove Theorem \ref{thm: maintheorem}, we will need to upgrade Lemma \ref{lem:Lnk_toprobab_0} 
from pointwise convergence to a more uniform mode of convergence.  By a lemma of Tao and Vu \cite[Lemma 3.1]{tao_vu}, it will be sufficient to show tightness of a second moment.  

\begin{lemma}\label{lem:tightness}
	For each $r>0$ and $k \leq \frac{\log n}{5 \log \log n}$, the sequence of random variables $$\left(\frac 1 {n^2} \int_{B_r} \left(\log|L_n^{(k)}(z)|\right)^2dm(z)\right)_{n \geq 1}$$ is tight.
\end{lemma}

We now deduce Theorem \ref{thm: maintheorem} from Lemmas \ref{lem:Lnk_toprobab_0} and \ref{lem:tightness}.

\begin{proof}[Proof of Theorem \ref{thm: maintheorem}]
It is sufficient to show that for every smooth compactly supported function $\varphi:\C \to \R$ we have
\begin{align*}
	\frac{1}{n}\sum\limits_{z\in \C: P_n^{(k)}(z)=0}\varphi(z)\to \int\limits_{\C}\varphi(z)\,d\mu(z)
\end{align*}
in probability as $n\to \infty$ (see, e.g., \cite[Theorem 14.16]{kallenberg}). By, e.g., \cite[Section 2.4.1]{HKPV}, we have
\begin{align}\label{eq:Laplacian-distribution-formula}
\frac{1}{n}\int_{\C} (\log |L_n^{(k)}(z)|)\Delta \varphi(z)\, dm(z)=  \frac{1}{n}\sum_{z\in \C: P_n^{(k)}(z)=0} \varphi(z)- \frac{1}{n}\sum_{z\in \C: P_n(z)=0} \varphi(z).
\end{align}

Combining Lemmas \ref{lem:tightness} and \ref{lem:Lnk_toprobab_0} with \cite[Lemma~3.1]{tao_vu} shows 	\begin{equation}\label{eq:lem_int_log_Lnk}
	\frac{1}{n}\int_{\C} (\log |L_n^{(k)}(z)|) \Delta \varphi(z) dm(z) \toprobab 0\,.
\end{equation}
By the law of large numbers, the right-most term of \eqref{eq:Laplacian-distribution-formula} tends to $\int\limits_{\C}\varphi(z)d\mu(z)$ almost-surely, completing the proof.
\end{proof}

\subsection{Proof of Lemma \ref{lem:Lnk_toprobab_0}}

\begin{lemma}\label{lem:Ai-zeroset}
		Let $F=\{z\in \C: \int_{\C}|y-z|^{-1}d\mu(y)=\infty\}$. Then $m(F) = 0$.
\end{lemma}
\begin{proof}
	Since $\mu$ is a probability measure we have that $\int\limits_{y:|y-z|\ge 1}|z-y|^{-1}d\mu(y)\le 1$ for all $z$.  Thus $F$ is equal to the set of $z$ for which $\int_{B_1(z)} |y - z|^{-1}\,d\mu(y) = \infty$.  Apply Fubini's theorem to compute
	\begin{align*}
	\int\limits_{\C}\int\limits_{B_1(z)}|z-y|^{-1}d\mu(y)dm(z)=	 \int\limits_{\C} \int\limits_{B_1(y)} |z-y|^{-1}dm(z)d\mu(y)=2\pi.
	\end{align*}
Thus $m(F)=0$.
\end{proof}

\begin{proof}[Proof of Lemma \ref{lem:Lnk_toprobab_0}]
	If the random variable $X_1$ is almost surely a constant, then the Lemma follows easily, and so we assume that $X_1$ is non-degenerate.
	
	First, we show that for every fixed $z\in \C\setminus F$  and every $\eps>0$, we have
	\begin{equation}\label{eq:Ln_upper_bound}
		\lim_{n\to\infty} \P\left[\frac{1}{n} \log |L_n^{(k)}(z)|\geq \eps\right]=0.
	\end{equation}
Indeed, defining $Y_i = (z - X_i)^{-1}$ for all $i \in [n]$, Markov's inequality bounds
\begin{align}\label{eq:Ln-Markov}
	\P\left[ |L_n^{(k)}(z)|\ge e^{\eps n}\right]\le e^{-\eps n} \E[|L_n^{(k)}(z)|] \leq e^{-\eps n} \binom{n}{k} \left(\E |Y_1|\right)^k\,.
\end{align}
Since $z \in \C \setminus F$, the expectation $\E |Y_1|$ is finite.  Thus, the right-hand-side tends to zero as long as $k = o(n)$ and so the limit \eqref{eq:Ln_upper_bound} follows.

	An application of Lemma \ref{lem:anti-concentration_multiVar} will provide a lower bound, i.e.,  we prove that for every $z\in \C$ which is not an atom of $\mu$ and every $\varepsilon>0$,
	\begin{equation}\label{eqn:Lnk_lower_bound}
		\lim_{n\to\infty} \P\left[\frac{1}{n} \log |L_n^{(k)}(z)|\leq -\eps\right]=0.
	\end{equation}

Write
\begin{align*}
\P\left[\frac{1}{n} \log |L_n^{(k)}(z)|\leq -\eps\right] =   \P\left(\left|\sum_{1\le i_1<i_2<\dots<i_k\le n }Y_{i_1}Y_{i_2}\dots Y_{i_k}\right|\le e^{-\eps n}\right)
\end{align*}
and note that Lemma \ref{lem:anti-concentration_multiVar} shows the right-hand-side is $o(1)$, completing the proof of the Lemma.
\end{proof}

\subsection{Proof of Lemma \ref{lem:tightness}}

To prove tightness, we will use the Poisson-Jensen formula to write $\log|L_n^{(k)}(z)|$ in terms of a Poisson integral of $\log|L_n^{(k)}(z)|$ along the circle of radius $R$ for some $R>r$, plus a small correction in terms of the zeros and poles of $L_n^{(k)}(z)$. From here, showing tightness of the $L^2$ norm will come in two steps: $\frac{1}{n}\log|L_n^{(k)}(z)|$ is tight at $0$ and on the circle of radius $R$, and that the correction depending on the zeros and poles is not too large.

Note that by the proof of Lemma \ref{lem:Ai-zeroset} we have that for any $R > 0$ we have that $$\int_{B_R} \int_{\C} |z - y|^{-1} d\mu(y) \,dm(z) < \infty\,.$$
By switching the outer integral into polar coordinates, we see that for Lebesgue-almost-all $R > 0$ we have that \begin{equation}\label{eq:R-circle-finite}
	\int_{0}^{2\pi} \int_{\C} \frac{1}{|R e^{i\theta} - y|}\,d\mu(y)\,d\theta < \infty\,.
\end{equation}
Throughout, we assume that $R$ satisfies this assumption.

We now write the Poisson-Jensen formula for $\log|L_n^{(k)}(z)|$: Let $R>r$ be chosen as above, let $x_{1,n},\ldots, x_{j_n,n}$ denote the zeros of $P_n(z)$ in the disk $B_R$, and let $y_{1,n},\ldots, y_{\ell_n, n}$  be the zeros of $P_n^{(k)}(z)$ in the disk $B_R$. Also note that $j_n\le n$ and $\ell_n\leq n$. We take the following standard facts about the Poisson-Jensen formula from \cite[Chapter 8]{markushevich}:

\begin{lemma}[Poisson-Jensen formula]
\begin{equation}\label{eq:poisson_jensen}
	\log |L_n^{(k)}(z)|
	=
	I_n^{(k)}(z;R)
	+\sum_{\ell=1}^{\ell_n} \log \left|\frac{R(z-y_{\ell,n})}{R^2-\bar y_{\ell,n}z}\right|
	-\sum_{j=1}^{j_n} \log \left|\frac{R(z-x_{j,n})}{R^2-\bar x_{j,n}z}\right|,
\end{equation}
where 
\begin{equation}\label{eq:poisson_integral}
	I_n^{(k)}(z;R)=\frac 1 {2\pi} \int_{0}^{2\pi} \log |L_n^{(k)}(R e^{i\theta})|P_R(|z|, \theta-\arg z)d\theta
\end{equation}
where $P_R$ denotes the Poisson kernel
\begin{equation}\label{eqn:P_Ker}
	P_R(r, \phi)=\frac{R^2-r^2}{R^2+r^2-2Rr \cos \phi},
	\;\;\;
	r\in[0,R],\, \phi\in[0,2\pi].
\end{equation}
\end{lemma}

The only property of the Poisson kernel we need is that it provides only bounded distortion when $r$ is uniformly bounded away from $R$.  
\begin{fact}\label{fact:P_R-bound}
	For all $0 < r < R$ there exists $M \geq 1$ so that for all $z \in B_r$ and $\theta \in [0,2\pi]$ we have 	\begin{equation}\label{eqn: P_R-bound}
		\frac{1}{M} \leq P_R(|z|,\theta) \leq M\,.  
	\end{equation}
\end{fact}

We first provide an upper bound $ I_n^{(k)}(z;R)$ in probability that is uniform in $B_r$ for all $k$ and $n$.

\begin{lemma}\label{step1:UpperBound_I}
For all $R$ satisfying \eqref{eq:R-circle-finite} and $0 < r < R$ there is a constant $C > 0$ so that for all $t > 0$ and $1 \leq k \leq n$ we have 
\begin{equation}
	\P\left(\frac{1}{n}\sup_{z \in B_r} I_n^{(k)}(z;R) \geq t \right)  \leq \frac{C}{t}\,.
\end{equation}
\end{lemma}
\begin{proof}
	Apply Fact \ref{fact:P_R-bound} and bound
	\begin{align}\label{eqn:bound-Ink}
			I_n^{(k)}(z;R) 
			&\leq \frac{M}{2 \pi} \int_0^{2\pi} \log_{+} \left| L_n^{(k)}(R e^{i\theta})\right| d\theta.
		\end{align}
	Bound \begin{align}
		\frac{1}{n}\log_{+} \left| L_n^{(k)}(R e^{i\theta})\right| &\leq \frac{1}{n}\log\left(1 +  |L_n^{(k)}(R e^{i\theta})| \right) \nonumber \\
		&\leq \frac{1}{n}\log\left(1 +  \binom{n}{k}^{-1}|L_n^{(k)}(R e^{i\theta})| \right) + 1 \label{eq:log+-bound}
	\end{align}
where  we used that $n^{-1} \log \binom{n}{k} \leq 1$. 

\begin{claim}\label{cl:Jensen}
	We have \begin{equation*}
		\int_{0}^{2\pi} \E  \frac{1}{n}\log\left(1 +  \binom{n}{k}^{-1}|L_n^{(k)}(R e^{i\theta})| \right)\,d\theta \leq \int_0^{2\pi} \int_{\C} |R e^{i\theta} - y|^{-1}\,d\mu(y)\,d\theta\,.
	\end{equation*}
\end{claim}
\begin{proof}[Proof of Claim \ref{cl:Jensen}]
	 Noting that the function $x \mapsto \log(1 + x)$ is concave for $x \geq 0$, Jensen's inequality bounds \begin{align*}
		\E \log\left(1 +  \binom{n}{k}^{-1}|L_n^{(k)}(R e^{i\theta})| \right) &\leq \log \left(1 +  \E\binom{n}{k}^{-1}|L_n^{(k)}(R e^{i\theta})|  \right)\,. 
	\end{align*}
	
	Since $R$ satisfies \eqref{eq:R-circle-finite}, for almost all $\theta \in [0,2\pi]$ we may bound $$\E\binom{n}{k}^{-1}|L_n^{(k)}(R e^{i\theta})| \leq \left( \int_{\C} \left|R e^{i\theta} - y \right|^{-1}\,d\mu(y) \right)^{k}\,.$$
	
	Combining the previous two displayed equations and integrating over $\theta$ shows that \begin{align*}\int_0^{2\pi}\frac{1}{n}\E\log\left(1 +  \binom{n}{k}^{-1}|L_n^{(k)}(R e^{i\theta})| \right)\,d\theta  &\leq \int_0^{2\pi} \frac{1}{n}\log\left(1 + \left( \int_{\C} \left|R e^{i\theta} - y \right|^{-1}\,d\mu(y) \right)^{k} \right)\,d\theta \\
		&\leq \int_0^{2\pi} \frac{k}{n} \int_{\C} |R e^{i\theta} - y|^{-1}\,d\mu(y)\,d\theta \\
		&\leq  \int_0^{2\pi}  \int_{\C} |R e^{i\theta} - y|^{-1}\,d\mu(y)\,d\theta 
	\end{align*}
	where we used the elementary inequality $\log(1 + x^k) \leq k x$ for $x \geq 0$.  
\end{proof}
Combining lines \eqref{eqn:bound-Ink} and \eqref{eq:log+-bound} with the Claim and Markov's inequality completes the proof.
\end{proof}

The Poisson-Jensen form will allow us to compare $I_n^{(k)}(z;R)$ to $I_n^{(k)}(0;R)$.  It will be convenient to assume that $0 \notin F$.  If $0\in F$, we may find some $a\notin F$ (recall that $F$ has measure $0$) and replace each $X_i$ with $X_i-a$; thus, we may assume without loss of generality that $0\notin F$. 

\begin{lemma} \label{Step2:Ink0R_lower_bound}
There is a constant $C_1$ depending only on $R>0$ such that for  $k \leq \frac{\log n}{5\log\log n}$ we have
\begin{equation}
	\lim_{n\to\infty} \P\left[\frac{1}{n}I_n^{(k)}(0;R)\leq  -C_1\right]=0.
\end{equation}

\end{lemma}

\begin{proof}
	By the Poisson--Jensen formula~\eqref{eq:poisson_integral} at $z=0$, we have
	\begin{equation}\label{eqn:poisson_jensen_at0}
		I_n^{(k)}(0;R)
		=
		\log |L_n^{(k)}(0)|
		-	 \sum_{\ell=1}^{\ell_n} \log \left|\frac{y_{\ell,n}}{R}\right|
		+	 \sum_{j=1}^{j_n} \log \left|\frac{x_{j,n}}{R}\right|.
	\end{equation}
	For the first term on the right-hand side of~\eqref{eqn:poisson_jensen_at0}, by \eqref{eqn:Lnk_lower_bound},
	we have
	\begin{equation}\label{eq:term1-Poisson}
		\lim_{n\to\infty}\P\left[\left(\frac{1}{n} \log |L_n^{(k)}(0)|\leq -1\right)\right]=0.
	\end{equation}
	For the second term on the right-hand side of~\eqref{eqn:poisson_jensen_at0}, we get a trivial bound
	\begin{equation}\label{eq:termY-poisson}
		\frac{1}{n}\sum_{\ell=1}^{\ell_n} \log \left|\frac{y_{\ell,n}}{R}\right| \leq 0.
	\end{equation}
	For the last term on the right-hand side of~\eqref{eqn:poisson_jensen_at0}, using the strong law of large numbers, we see
	\begin{equation}\label{eq:LLN_for_log_zeros}
		\frac{1}{n} \sum_{j=1}^{j_n} \log \left|\frac{x_{j,n}}{R}\right| = -\frac{1}{n} \sum_{j=1}^{n} \log_{-} \left|\frac{X_{j}}{R}\right| \toas -\E \left[\log_- \left|\frac{X_1}{R}\right|\right].
	\end{equation}
	Observe that $z\mapsto \log_-|z/R|-\log_-|z|$ is a bounded function with compact support and $\E \log_-|X_1| <\infty$ by the assumption $0\notin F$, so we have $\E \left[\log_- \left|\frac{X_1}{R}\right|\right]<\infty$.  Combining  \eqref{eqn:poisson_jensen_at0}, \eqref{eq:term1-Poisson}, \eqref{eq:termY-poisson} and \eqref{eq:LLN_for_log_zeros} completes the proof.
\end{proof}

To translate Lemma \ref{Step2:Ink0R_lower_bound} to a uniform lower bound for $z\in B_r$, we will use the fact that the Poisson kernel has bounded distortion, and so it is sufficient to consider the value at $0$ and the boundary. 

 \begin{lemma}\label{Step3: nkzR_lower_bound}
For all $0 < r < R$ there is a constant $C$ such that for all $n$, $t > 0$ and $k \leq \frac{\log n}{5 \log \log n}$ we have 
\begin{equation}
	\P\left[\frac{1}{n}\inf_{z\in B_r} I_n^{(k)}(z;R) \leq -t \right] \leq \frac{C}{t} .
\end{equation}
 \end{lemma}
 \begin{proof}
 	By Fact \ref{fact:P_R-bound} we may lower bound
 	\begin{align}
 		\frac {2\pi} n I_n^{(k)}(z;R)
 		&=\frac{1}{n}\int_{0}^{2\pi} \log |L_n^{(k)}(R e^{i\theta})|P_R(|z|, \theta-\arg z)d\theta \nonumber\\
 		&=\frac{1}{n}	\int_{0}^{2\pi} \log_{+} |L_n^{(k)}(R e^{i\theta})|P_R(|z|, \theta-\arg z)d\theta-\frac{1}{n}\int_{0}^{2\pi} \log_{-} |L_n^{(k)}(R e^{i\theta})|P_R(|z|, \theta-\arg z)d\theta \nonumber \\
 		&\ge \frac{1}{Mn} \int_0^{2\pi}\log_{+} |L_n^{(k)}(R e^{i\theta})| d\theta- \frac{M}{n}\int_0^{2\pi} \log_{-} |L_n^{(k)}(R e^{i\theta})| d\theta \nonumber \\
 		&=
 		\frac {2\pi M} n I_n^{(k)}(0;R)-\frac{1}{n}  \left(M-\frac {1}{M}\right) \int_0^{2\pi}  \log_{+} |L_n^{(k)}(R e^{i\theta})|d\theta. \label{eq:I_nz-LB}
 	\end{align}

By Lemma \ref{Step2:Ink0R_lower_bound}, we have that $\frac {2\pi M} n I_n^{(k)}(0;R) \geq -C$ asymptotically almost surely.  By the proof of Lemma \ref{step1:UpperBound_I} we have $$\P\left( \frac{1}{n}\int_0^{2\pi}  \log_{+} |L_n^{(k)}(R e^{i\theta})|d\theta \geq t \right) \leq \frac{C}{t}\,.$$

Combining these two bounds with \eqref{eq:I_nz-LB} completes the Lemma.
 \end{proof}

\begin{proof}[Proof of Lemma \ref{lem:tightness}]
Applying the Cauchy-Schwarz inequality to \eqref{eq:poisson_jensen} and dividing both sides by $n^2$, we have
\begin{align} \label{eq:poisson_jensen_sq_ineq}
	\frac 1{n^2}\log^2 |L_n^{(k)}(z)|
	\leq
	\frac{3}{n^2} I_n^{(k)}(z;R)^2
	+\frac{3 \ell_n}{n^2}\sum_{\ell=1}^{\ell_n} \log^2 \left|\frac{R(z-y_{\ell,n})}{R^2-\bar y_{\ell,n}z}\right|
	+\frac{3 j_n}{n^2} \sum_{j=1}^{j_n} \log^2 \left|\frac{R(z-x_{j,n})}{R^2-\bar x_{j,n}z}\right|.
\end{align}
Note that for any $y\in B_R$ and $z\in B_r$ we have $|R^2-\bar{y}z|$ is uniformly bounded below. Using integrability of $(\log|z|
)^2$ near $0$, this implies that
\begin{equation*}
	\sup_{y\in B_R} \int\limits_{B_r} \log^2 \left|\frac{R(z-y)}{R^2-\bar y z}\right|dm(z) + \sup_{x\in B_R} \int\limits_{B_r} \log^2 \left|\frac{R(z-x)}{R^2-\bar x z}\right|dm(z)  \le C
\end{equation*}
for some constant $C$. Bounding $j_n\le n$ and $\ell_n\leq n$ shows a deterministic bound of 
\begin{equation*}
	\frac{3 \ell_n}{n^2}\sum_{\ell=1}^{\ell_n} \int\limits_{B_r} \log^2 \left|\frac{R(z-y_{\ell,n})}{R^2-\bar y_{\ell,n}z}\right|dm(z)
	+\frac{3 j_n}{n^2} \sum_{j=1}^{j_n} \int\limits_{B_r} \log^2 \left|\frac{R(z-x_{j,n})}{R^2-\bar x_{j,n}z}\right|dm(z)
	\leq 3C.
\end{equation*}
Lemma \ref{step1:UpperBound_I}  and \ref{Step3: nkzR_lower_bound} show that $\frac{1}{n^2}\int\limits_{B_r}(I_n^{(k)}(z;R))^2 dm(z)$ is tight. Since a tight sequence plus a deterministically bounded sequence is tight, the proof is complete.
\end{proof}

\section{Anticoncentration of elementary symmetric polynomials}
We will deduce Lemma \ref{lem:anti-concentration_multiVar} from a Theorem of Meka, Nguyen and Vu concerning anti-concentration of multi-affine polynomials of Bernoulli random variables.  To properly state their theorem, we will need a bit of setup.  

We will consider multi-affine polynomials of the form \begin{equation}\label{eq:form-poly}
	Q(z_1,\ldots,z_n) = \sum_{S \subset [n]} a_S \prod_{j \in S} z_j
\end{equation}
where the coefficients $a_S$ are real and the \emph{degree} of $Q$ is the largest  $|S|$ so that $a_{|S|} \neq 0$.  If $Q$ is of degree $d$, then the \emph{rank} of $Q$ is the largest integer $r$ such that there exist disjoint sets $S_1,\ldots,S_r \subset [n]$ of size $d$ with $|a_{S_j}| \geq 1$ for all $j \in [r]$.

\begin{theorem}[Theorem 1.7 of \cite{Meka-Ng-Vu2016}]\label{th:littlewoodoffordbiasedmain}
	There is an absolute constant $B$ such that the following holds.
	Let $Q$ be a polynomial of the form \eqref{eq:form-poly} whose rank $r \ge 2$. Let $p$ be such that $\tilde r:= 2^{d}\alpha^{d}r=2^{d}\alpha^{d}n/k\ge 3$ where $\alpha := \min\{p, 1-p\}$. Then for any interval $I$ of length 1, if we let $\{\eps_j\}_{j = 1}^n$ be i.i.d.\ Bernoulli$(p)$ random variables we have
	$$\P( Q(\eps_1,\ldots,\eps_n) \in I) \leq \frac{B d^{4/3}(\log \tilde r )^{1/2}}{(\tilde r )^{1/(4d+1)}}\,.$$
\end{theorem}

To apply this Theorem to the case of Lemma \ref{lem:anti-concentration_multiVar}---where the variables are complex valued and arbitrary---we will use the non-degeneracy to first sample most of the randomness of each $Y_j$, leaving a Bernoulli random variable's worth of randomness behind; then we will take either the real or imaginary part of the subsequent polynomial and show that it has sufficiently high rank with high probability.

\begin{proof}[Proof of Lemma \ref{lem:anti-concentration_multiVar}]
	Define
\begin{equation}
	P (Y_1, \dots, Y_n)  := \sum_{ S \subset \{1, \dots, n\}; |S|=k } \prod_{j \in S} Y_j. \label{form}
\end{equation}

Since $Y_j$ is non-degenerate we have that at least one of $\Re Y_j$ or $\Im Y_j$ is non-degenerate.  By  replacing each $Y_j$ with $i Y_j$ if needed, we assume without loss of generality that $\Re Y_j$ is non-degenerate.  As such, there exist $t \in \R, \kappa > 0, q \in (0,1)$ so that \begin{equation}\label{eq:Y-nondeg}
	\P(\Re Y_j - t > \kappa ) \geq q \quad \text{ and } \quad	\P(\Re Y_j - t < -\kappa ) \geq q \,.
\end{equation}

Note that since $k = O(\log n / \log\log n)$ we may bound $\kappa^k = o(e^{\eps n/2})$, and so by replacing $Y_j$ with $Y_j / \kappa$, we may assume without loss of generality that $\kappa = 1$.  

For each $j=1, \dots, n$, let $Y_j^{+}$ and $Y_j^{-}$ be independent random variables satisfying $\P(Y_j^{+}\in K) = \P(Y_j - t\in K|\Re Y_j  - t>0)$ and $\P(Y_j^{-}\in K) = \P(Y_j - t\in K|\,\Re Y_j - t \le 0)$ for all measurable subsets $K\subset \C$. 
Let $\eta_j=\mathbbm{1}_{\Re Y_j-t\geq 0},\, j\in[n]$, and note that the collection $\{\eta_j\}$ are i.i.d.\ Bernoulli random variables with parameter $p := \P(\Re Y_j - t > 0)$.
Let $Y_j' = \eta_{j}Y_j^{+} + (1-\eta_{j})Y_j^{-} +t$, and observe that $Y_j'$ and $Y_j$ have the same distribution.  Therefore, it suffices show that $$\P\left(|P(Y_1', \dots, Y_n')| \leq 1/2\right) = o(1)\,.$$  

For each fixed instance of the variables $\{Y_j^+, Y_j^-\}$, define the polynomial $F$ in the variables $\{\eta_j\}$ via \begin{align*}
	F(\eta_1,\ldots,\eta_n) &:=  P(\eta_1(Y_1^{+}-Y_1^{-}) + Y_1^{-} + t, \dots, \eta_n(Y_n^{+}-Y_n^{-}) + Y_n^{-} + t) \\
	&= \sum_{S \subset [n], |S| = k} \left(\prod_{j \in S} (Y_j^+ - Y_j^-) \right) \prod_{j \in S} \eta_j + Q
\end{align*}

where $Q$ is a polynomial of degree $<k$ in terms of $\eta_j$ when all the $Y_j^{\pm}$ are fixed. For $r = \lfloor n/k \rfloor$, let $S_1,\ldots,S_r$ be disjoint subsets of $[n]$ of size $k$.  For $S \subset [n] $, define $b_S := \prod_{j \in S}(Y_j^+ - Y_j^-)$.  

\begin{claim} \label{cl:big-b-coeffs}
	With probability at least $1 - 2 \exp(-c q^{2k}n/k)$, there are at least $\frac{q^{2k}n}{2k}$ coefficients $b_{S_{\ell}}$  satisfying $\big|b_{S_{\ell}}\big|\ge 2^{k}$.
\end{claim}
\begin{proof}[Proof of Claim \ref{cl:big-b-coeffs}]
	For $\ell \in [r]$, define the event $\cE_\ell = \{|b_{S_\ell}| \geq 2^k\}$, and note that the events $\{\cE_{\ell}\}_{\ell}$ are independent.  Further, note that   \begin{align}
		\P(\cE_\ell)&\ge 	\P(\big| Y_j^{+}-Y_j^{-}\big|\ge 2, \forall j\in S_{\ell})\notag\\
		&\ge \P(\Re  Y_j^{+}-\Re Y_j^{-}\ge 2\kappa , \forall j\in S_{\ell})\notag\\
		&\ge\P(\left(\Re  Y_j^{+}-t\right)-\left(\Re Y_j^{-}-t\right)\ge 2 , \forall j\in S_{\ell})\notag\\
		&\ge q^{2k} 
	\end{align}
by \eqref{eq:Y-nondeg}, since we have assumed $\kappa = 1$.  Let $N$ denote the number of $\ell$ for which $\cE_\ell$ holds.  Then by Bernstein's inequality there is a constant $c > 0$ so that $$\P(N \leq q^{2k} r / 2) \leq \P(N \leq \E N / 2) \leq 2 \exp(-c \E N) \leq 2 \exp\left( -cq^{2k} n/k \right)\,.$$
Noting that $\E N \geq q^{2k} n / k$ completes the claim.
\end{proof}

Note that $$\Re F(\eta_1,\ldots,\eta_n) = \sum_{S\subset [n], |S| = k } \left(\Re b_{S}\right) \prod_{j \in S} \eta_j + \Re Q, \quad \Im F(\eta_1,\ldots,\eta_n) = \sum_{S\subset [n], |S| = k } \left(\Im b_{S}\right) \prod_{j \in S} \eta_j + \Im Q  $$

and so $\Re F$ and $\Im F$ are both multi-affine polynomials in the variables $\{\eta_j\}$ with real coefficients and degree at most $k$.  For each coefficient $|b_{S_\ell}|$ with $|b_{S_\ell}| \geq 2^{k}$ we have either $|\Re b_{S_\ell}| \geq 2^{k-1}$ or $|\Im b_{S_\ell}| \geq 2^{k-1}$.  Thus, if we let $\cA$ be the event that either $\Re F$ or $\Im F$ has rank at least $q^{2k} n /(4k)$, then Claim \ref{cl:big-b-coeffs} implies that $\P(\cA^c) \leq 2 \exp(-c q^{2k}n/k)$.

Thus, for $n$ sufficiently large, we have \begin{align}
	\P(|P(Y_1,\ldots,Y_n)| \leq e^{-\eps n} ) &\leq  \P\left( |\Re F(\eta_1,\ldots,\eta_n)| \leq \frac{1}{2},   |\Im F(\eta_1,\ldots,\eta_n)| \leq \frac{1}{2}  \,\big|\, \cA\right) + 2 e^{-c q^{2k}n/k} \nonumber \\ 
	&=  \P\left( |\Re F(\eta_1,\ldots,\eta_n)| \leq \frac{1}{2},   |\Im F(\eta_1,\ldots,\eta_n)| \leq \frac{1}{2}  \,\big|\, \cA\right) + o(1)
\end{align}
where for the second bound we used that $k = o(\log n)$.

Conditioned on the event $\cA$, suppose that the rank of $\Re F$ is at least $q^{2k}n / (4k)$.  Then we will apply the first bound in Theorem \ref{th:littlewoodoffordbiasedmain}; if we set $\alpha = \min\{p,1-p\}$ and $\widetilde{r} = 2^k \alpha^k q^{2k}n / (4k)$ then Theorem \ref{th:littlewoodoffordbiasedmain} implies \begin{equation}
	\P\left( |\Re F(\eta_1,\ldots,\eta_n)| \leq 1/2 \,\Big|\, \rank \Re F(\eta_1,\ldots,\eta_n) \geq q^{2k}n / (4k)\right) \leq B k^{4/3} \frac{(\log \widetilde{r})^{1/2}}{\widetilde{r}^{1/(4k+1)}}\,.
\end{equation}

Write $k = \eps \log n / \log \log n$, where $\eps \leq 1/5$ and note we first note that $$\log \widetilde{r} = \log n\left(1 + \Theta\left(\frac{\eps}{\log \log n} \right) \right)\,.$$

and so 
\begin{align*}\log \left(B k^{4/3} \frac{(\log \widetilde{r})^{1/2}}{\widetilde{r}^{1/(4k+1)}} \right) &\leq \frac{1}{2} \log \left(\log n\left(1 + \Theta\left(\frac{\eps}{\log \log n}\right) \right) \right)  - \frac{\log \log n}{8\eps} + \Theta\left(\log \log \left( \eps \frac{\log n}{\log \log n} \right)\right) \\
	&=\left(\frac{1}{2}-  \frac{1}{8\eps}\right)\log \log n + o(1) + o(\log \log n)\,.
\end{align*}

For $\eps \leq 1/5$, the right-hand side  tends to negative infinity, completing the proof.
 \end{proof}

\section{Comments and Open Problems}
We believe that Lemma \ref{lem:anti-concentration_multiVar} is suboptimal, and that in fact the same statement should hold provided $k = o(n)$. 

\begin{conjecture}\label{conj:anti-con}
		Let $Y_1, Y_2, \dots, $ be i.i.d.\ copies of a complex-valued non-degenerate random variable. Then for any sequence $k = k(n)$ satisfying $k = o(n)$ and each $\eps>0$ we have 
		\begin{align*}
			\lim_{n\to  \infty}\P\left(\left|\sum_{1\le i_1<i_2<\dots<i_k\le n }Y_{i_1}Y_{i_2}\dots Y_{i_k}\right|\le e^{-\eps n}\right)=0.
		\end{align*}
\end{conjecture}

The sharpest general statements for anti-concentration for multi-affine polynomials are provided by the work of Meka, Nguyen and Vu \cite{Meka-Ng-Vu2016}, for which Conjecture \ref{conj:anti-con} lies well beyond the currently known bounds.  For an example of a polynomial of very high degree for which anti-concentration is known, see Tao and Vu's work on the permanent of a random matrix \cite{tv-permanent} (see also \cite{kwan-sauermann}).

\begin{remark}\label{remark}
	We note that substituting a positive resolution to Conjecture \ref{conj:anti-con} for Lemma \ref{lem:anti-concentration_multiVar} would immediately upgrade Theorem \ref{thm: maintheorem} to hold for all $k = o(n)$; indeed, throughout the proof the only assumptions on $k$ used are that $k = o(n)$ and that the conclusion of Lemma \ref{lem:anti-concentration_multiVar} holds.  Additionally, due to Lemma \ref{lem:Ai-zeroset}, one can further assume that the variable $Y_1$ in Conjecture \ref{conj:anti-con} satisfies $\E |Y_1| < \infty$.  
\end{remark}

We also highlight a conjecture of Kabluchko\footnote{This conjecture was stated at the Workshop on Random Functions in April, 2021.}:
\begin{conjecture}[Kabluchko] \label{conj:AS}
	For any probability measure $\mu$ on $\C$, the sequence of random measures $\mu_n^{(1)}$ converges to $\mu$ almost surely as $n \to \infty$.
\end{conjecture}

We suspect that Theorem \ref{thm: maintheorem} can also be upgraded to almost sure convergence.  The main difficulty appears to be in upgrading Lemma \ref{lem:Lnk_toprobab_0} to almost-sure convergence.

\section*{Acknowledgments}
Both authors supported in part by NSF grant DMS-2137623.

\bibliographystyle{abbrv}
\bibliography{kpoly}

\begin{thebibliography}{10}

\bibitem{alazard2022dynamics}
T.~Alazard, O.~Lazar, and Q.~H. Nguyen.
\newblock On the dynamics of the roots of polynomials under differentiation.
\newblock {\em Journal de Math{\'e}matiques Pures et Appliqu{\'e}es},
  162:1--22, 2022.

\bibitem{byun-lee-reddy}
S.-S. Byun, J.~Lee, and T.~R. Reddy.
\newblock Zeros of random polynomials and their higher derivatives.
\newblock {\em Trans. Amer. Math. Soc.}, 375(9):6311--6335, 2022.

\bibitem{hoskins2021dynamics}
J.~Hoskins and Z.~Kabluchko.
\newblock Dynamics of zeroes under repeated differentiation.
\newblock {\em Experimental Mathematics}, pages 1--27, 2021.

\bibitem{hoskins2022semicircle}
J.~G. Hoskins and S.~Steinerberger.
\newblock A semicircle law for derivatives of random polynomials.
\newblock {\em International Mathematics Research Notices},
  2022(13):9784--9809, 2022.

\bibitem{HKPV}
J.~B. Hough, M.~Krishnapur, Y.~Peres, and B.~Vir\'{a}g.
\newblock {\em Zeros of {G}aussian analytic functions and determinantal point
  processes}, volume~51 of {\em University Lecture Series}.
\newblock American Mathematical Society, Providence, RI, 2009.

\bibitem{Kabluchko2015}
Z.~Kabluchko.
\newblock Critical points of random polynomials with independent identically
  distributed roots.
\newblock {\em Proc. Amer. Math. Soc.}, 143(2):695--702, 2015.

\bibitem{kabluchko2021repeated}
Z.~Kabluchko.
\newblock Repeated differentiation and free unitary {P}oisson process.
\newblock {\em arXiv preprint arXiv:2112.14729}, 2021.

\bibitem{kallenberg}
O.~Kallenberg.
\newblock {\em Foundations of modern probability}.
\newblock Probability and its Applications (New York). Springer-Verlag, New
  York, 1997.

\bibitem{kiselev2022flow}
A.~Kiselev and C.~Tan.
\newblock The flow of polynomial roots under differentiation.
\newblock {\em Annals of PDE}, 8(2):1--69, 2022.

\bibitem{kwan-sauermann}
M.~Kwan and L.~Sauermann.
\newblock On the permanent of a random symmetric matrix.
\newblock {\em Selecta Math. (N.S.)}, 28(1):Paper No. 15, 29, 2022.

\bibitem{markushevich}
A.~I. Markushevich.
\newblock {\em Theory of functions of a complex variable. {V}ol. {II}}.
\newblock Revised English edition translated and edited by Richard A.
  Silverman. Prentice-Hall, Inc., Englewood Cliffs, N.J., 1965.

\bibitem{Meka-Ng-Vu2016}
R.~Meka, O.~Nguyen, and V.~Vu.
\newblock Anti-concentration for polynomials of independent random variables.
\newblock {\em Theory Comput.}, 12:Paper No. 11, 16, 2016.

\bibitem{o2021nonlocal}
S.~O’Rourke and S.~Steinerberger.
\newblock A nonlocal transport equation modeling complex roots of polynomials
  under differentiation.
\newblock {\em Proceedings of the American Mathematical Society},
  149(4):1581--1592, 2021.

\bibitem{pemantle_rivin}
R.~Pemantle and I.~Rivin.
\newblock The distribution of zeros of the derivative of a random polynomial.
\newblock In {\em Advances in combinatorics}, pages 259--273. Springer,
  Heidelberg, 2013.

\bibitem{subramanian}
S.~D. Subramanian.
\newblock On the distribution of critical points of a polynomial.
\newblock {\em Electron. Commun. Probab.}, 17:no. 37, 9, 2012.

\bibitem{tv-permanent}
T.~Tao and V.~Vu.
\newblock On the permanent of random {B}ernoulli matrices.
\newblock {\em Adv. Math.}, 220(3):657--669, 2009.

\bibitem{tao_vu}
T.~Tao and V.~Vu.
\newblock Random matrices: universality of {ESD}s and the circular law.
\newblock {\em Ann. Probab.}, 38(5):2023--2065, 2010.
\newblock With an appendix by Manjunath Krishnapur.

\end{thebibliography}
\end{document}